\documentclass[a4paper,12pt]{article}

\usepackage[english]{babel}

\usepackage[top=3cm,bottom=3cm,left=2.5cm,right=2.5cm]{geometry}

\usepackage{amsmath}
\usepackage{graphicx}
\usepackage{amsmath,amsthm,mathrsfs}
\usepackage{bm}
\usepackage{amssymb}
\usepackage{cases}
\usepackage{hyperref}
\usepackage{arydshln}
\usepackage[english]{babel}
\usepackage{amsmath,amsthm}
\usepackage{amsfonts}
\usepackage{marvosym}
\usepackage{latexsym}
\usepackage{graphicx}
\usepackage{txfonts}
\usepackage[numbers,sort&compress]{natbib}
\usepackage[natural]{xcolor}
\usepackage{rotating}
\usepackage{mathtools}
\usepackage{enumerate}
\allowdisplaybreaks [4]

\newtheorem{theorem}{\bf Theorem}[section]
\newtheorem{lemma}[theorem]{Lemma}
\newtheorem{problem}[theorem]{Problem}

\newtheorem{claim}[theorem]{\indent Claim}

\begin{document}
\title{Max-Bisections of graphs without  even cycles\thanks{Research was supported by National Key R\&D Program of China (Grant No. 2023YFA1010202), National Natural Science Foundation of China (Grant No. 12371342), the Central Guidance on Local Science and Technology Development Fund of Fujian Province (Grant No. 2023L3003)}}
\author{Jianfeng Hou\footnote{Email: jfhou@fzu.edu.cn}, ~Siwei Lin\footnote{Email: linsw0710@163.com}, ~Qinghou Zeng\footnote{Email: zengqh@fzu.edu.cn}\\
{\small Center for Discrete Mathematics, Fuzhou University, Fujian, 350003, China}}
\date{}

\maketitle

\begin{abstract}
For an integer $k\ge 2$, let $G$ be a graph with $m$ edges and without cycles of length $2k$. The pivotal Alon-Krivelevich-Sudakov Theorem on Max-Cuts states that $G$ has a bipartite subgraph with at least $m/2+\Omega(m^{(2k+1)/(2k+2)})$ edges. In this paper, we present a bisection variant of it by showing that if $G$ has minimum degree at least $k$, then $G$ has a balanced bipartite subgraph with at least $m/2+\Omega(m^{(2k+1)/(2k+2)})$ edges. It not only answers a problem of Fan, Hou and Yu in full generality but also enhances a recent result given by Hou, Wu and Zhong.

Our approach hinges on a key bound for bisections of graphs with sparse neighborhoods concerning the degree sequence. The result is inspired by the celebrated approximation algorithm of Goemans and Williamson and appears to be worthy of future exploration.
\end{abstract}


\section{Introduction}
\subsection{Max-Cuts}
Let $G$ be a graph. A \emph{bipartition} (or a \emph{cut}) of $G$, denoted by $(A,B)$,  is a partition of $V(G)$ with $V(G)=A\cup B$ and $A\cap B=\emptyset$. If it satisfies $||A|-|B||\leq 1$, then we call it a \emph{bisection}. The $size$ of $(A,B)$, denoted by $e(A,B)$, is the number of edges with one end in $A$ and the other in $B$. The famous \emph{Max-Cut problem} is to find a cut $(A,B)$ of $G$ that maximizes $e(A,B)$.

Considering a random cut of a graph $G$ with $m$ edges, it is easy to see that $G$ has a cut of size at least $m/2$.
In 1967, Erd\H os \cite{Erd1967} demonstrated that the factor 1/2 cannot be improved.  Answering a question of Erd\H{o}s, Edwards \cite{Ed1973,Ed1975} proved that every graph with $m$ edges has a cut of size at least $m/2+(\sqrt{8m+1}-1)/8$, and this bound is tight for all complete graphs of odd order. For certain ranges of $m$, Alon \cite{A1996} gave an additive improvement of order $m^{1/4}$.

A natural next step is to explore Max-Cuts in graphs that are free of specific substructures. For a family $\mathscr{H}$ of graphs, we say a graph $G$ is $\mathscr{H}$-\emph{free} if $G$ contains no member of $\mathscr{H}$. If $\mathscr{H}=\{H\}$, then we write $H$-free rather than $\{H\}$-free for brevity. Denote by $C_\ell$ a cycle with $\ell$ vertices.  Erd\H{o}s and Lov\'{a}sz (see \cite{Erd1979}) initialed the topic by showing that every triangle-free graph with $m$ edges has a cut of size at least $m/2+\Omega(m^{2/3})$. Following earlier works (\cite{PT1994,She1992}), Alon \cite{A1996} later strengthened the bound to $m/2+\Theta(m^{4/5})$. In 2003, Alon, Bollob\'as, Krivelevich and Sudakov \cite{Alo2003} conducted a systematic investigation of Max-Cuts of graphs without short cycles and showed that every graph with $m$ edges and girth $r$ (the length of the shortest cycle) has a cut of size at least $m/2+\Omega(m^{r/(r+1)})$. Even when triangles are permitted, Zeng and Hou \cite{Zen2018} established that the same bound holds. Remarkably, the additional restriction forbidding
$C_3,\ldots, C_{r-1}$ turns out to be superfluous. The pivotal Alon-Krivelevich-Sudakov Theorem \cite{AKS2005} states that
\begin{theorem}[Alon, Krivelevich and Sudakov \cite{AKS2005}]\label{THM:AKS-thm}
For an integer $k\ge 2$, every $C_{2k}$-free graph with $m$ edges has a cut of size  at least
\begin{align}\label{THM:AKS-bound}
m/2+\Omega(m^{(2k+1)/(2k+2)}).
\end{align}
The bound is tight for $2k\in\{4,6,10\}$.
\end{theorem}
In the same paper, they conjectured that the bound $m/2+\Omega(m^{(r+1)/(r+2)})$ remains valid for $C_{r}$-free graphs with odd $r$. Following multiple intermediate steps (\cite{FHM2023,Zen2018}), it was confirmed  by Glock, Janzer and Sudakov \cite{Glock2023} using combining techniques from semidefinite programming, probabilistic reasoning, as well as combinatorial and spectral arguments.

\subsection{Max-Bisections of $C_{2k}$-free graphs}

In this paper, we focus on the \emph{Max-Bisection problem}: find a bisection of a given graph that maximizes its size. As noticed by Bollob\'as and Scott \cite{Bol2002}, Max-Bisection problems are very different from Max-Cut problems. For example, the Edwards' bound implicitly implies that a connected graph $G$ with $n$ vertices and $m$ edges admits a cut of size at least $m/2+(n-1)/4$ (see e.g. \cite{Erd1997,PT1994}). However,  every bisection of the complete bipartite graph $K_{d,n-d}$ with $m$ edges has size at most $\lceil (m+d^2)/2 \rceil$. Motivated by Max-Cuts, Bollob\'as and Scott \cite{Bol2002} asked the following problem.
\begin{problem}[Bollob\'as and Scott \cite{Bol2002}]\label{B-S-Problem}
What are the largest and smallest cuts that we can guarantee with bisections of graphs?
\end{problem}

Note that for any sufficiently small $\epsilon>0$, $K_{d,n-d}$ implies that we cannot guarantee a bisection of size $m/2+\Omega(m^{\epsilon})$ for $H$-free graphs with $m$ edges if $H$ contains an odd cycle. Thus, forbidden even cycle seems reasonable for Problem \ref{B-S-Problem}.
The majority of results are to find a bisection of size at least $m/2+cn$ for some $c>0$ in certain graphs with $n$ vertices and $m$ edges (see e.g. \cite{jin2019, LLS2013,Ji2019,Xu2010,Xu2010balanced,Xu2014}).
It would be natural to explore a bisection-type extension of the Alon-Krivelevich-Sudakov Theorem.

Let $G$ be a connected $C_4$-free graph with $n$ vertices, $m$ edges and minimum degree at least 2. Fan, Hou and Yu \cite{FHY2018} proved that $G$ admits a bisection of size at least $m/2+cn.$  This, along with the classical Bondy-Simonovits Theorem \cite{BS1974}, yields a bisection of $G$ with size at least $m/2+\Omega(m^{2/3})$. The following problem was initially posed in \cite{FHY2018} and subsequently explicitly stated as an open problem in \cite{HWY2024}.
\begin{problem}[\cite{FHY2018, HWY2024}]\label{Prob:Fan-Hou-Yu}
Find the largest exponent $\epsilon$ such that every $C_4$-free graph with $m$ edges and minimum degree at least 2 has a bisection of size at least $m/2+\Omega(m^{\epsilon})$.
\end{problem}

Using a bisection version of Shearer's randomized algorithm \cite{She1992},  Lin and Zeng \cite{Lin2021} showed that for any integer $k \ge3$, every $\{C_4,C_6,C_{2k}\}$-free graph with a perfect
matching admits a bisection satisfying \eqref{THM:AKS-bound}. In the same paper, they conjectured that the existence of a perfect matching could be ensured solely through a minimum degree condition, which was confirmed by Hou, Wu and Zhong \cite{HWY2024} recently.

\begin{theorem}[Hou, Wu and Zhong \cite{HWY2024}]\label{the1.2}
For any fixed integer $k \geq 3$, let $G$ be a connected $\{C_4,C_6,C_{2k}\}$-free graph with $m$ edges and minimum degree at least 2. Then $G$ admits a bisection of size at least $m/2+\Omega(m^{(2k+1)/(2k+2)})$.
\end{theorem}

Our main contribution is a bisection variant of the Alon-Krivelevich-Sudakov Theorem, which not only answers Problem \ref{Prob:Fan-Hou-Yu} in full generality but also enhances Theorem \ref{the1.2}.

\begin{theorem}\label{max bisection for c2kfree}
Given an integer $k\ge2$, let $G$ be a $C_{2k}$-free graph with $n$ vertices, $m$ edges and minimum degree at least $k$. Then $G$ has a bisection of size at least
\begin{align*}
m/2+\Omega(m^{(2k+1)/(2k+2)}).
\end{align*}
The bound is tight for $2k\in\{4,6,10\}$.
\end{theorem}

Considering bisections of complete bipartite graphs $K_{k-1,n-k+1}$, we remark that the minimum degree condition in Theorem \ref{max bisection for c2kfree} is tight. We will prove Theorem \ref{max bisection for c2kfree} when $k\ge 3$ in Section \ref{Bisections of C_2k-free graphs} and when $k=2$ in Section \ref{Bisections of C_4-free graphs}.



\subsection{Max-Bisections of graphs with sparse neighborhoods}

The degree sequence plays a key role in the Max-Cut problem. It is initialed by Shearer \cite{She1992} who showed that every triangle-free graph with $n$ vertices, $m$ edges and degree sequence $d_{1}, d_{2},\ldots ,d_{n}$ admits  a bipartition of size at least
\begin{equation}\label{sh-bound}
\frac m2+\Omega\left(\sum_{i=1}^{n}\sqrt{d_{i}}\right).
\end{equation}
Usually, the lower bound \eqref{sh-bound}, known as \emph{Shearer's bound}, serves as a pivotal foundation that has facilitated the derivation of several intriguing results in this field. A direct corollary, given by Shearer \cite{She1992} and remarked by Alon \cite{A1996}, is that if a graph with $m$ edges has a cut satisfying Shearer's bound, then it also has a cut of size at least $m/2+\Omega(m^{3/4})$. In fact, the Shearer's bound, augmented by supplementary techniques, typically produces stronger bounds for Max-Cuts in $H$-free graphs. Notable instances include the work of Alon, Krivelevich and Sudakov  \cite{AKS2005} for graphs with sparse neighborhoods.

\begin{theorem}[Alon, Krivelevich and Sudakov \cite{AKS2005}]\label{THM:AKS-theorem-sparse}
Let $G$ be a graph with $n$ vertices, $m$ edges and vertex degrees $d_1, d_2,\dots, d_n$. If there is an absolute positive constant $\epsilon$ such that the induced subgraph on any set of $d$ vertices, all of which have a common neighbor, contains at most $\epsilon d^{3/2}$ edges, then $G$ has a cut with size at least
\[
\frac m2+\Omega\left(\sum_{i=1}^n\sqrt{d_i}\right).
\]
\end{theorem}
The proof of Theorem \ref{THM:AKS-thm} in \cite{AKS2005} is a combination of Theorem \ref{THM:AKS-theorem-sparse} and a probabilistic reasoning. It is natural to seek sufficient conditions under which the Max-Bisections satisfy the Shearer's bound. The first step was given by Lin and Zeng \cite{Lin2021} who showed that every  $\{C_{4},C_{6}\}$-free graph  with a perfect matching has a bisection satisfying the Shearer's bound. Recently, Hou, Wu and Zhong \cite{HWY2024} showed that the perfect matching condition can be replaced by the minimum degree condition. For more information on this area, we refer the reader to \cite{HX2025, RHZ2022, WX2024, WZ2025}. In this paper, we give a bisection version of Theorem \ref{THM:AKS-theorem-sparse}.

\begin{theorem}\label{max bisection}
Let $G$ be a graph with $n$ vertices, $m$ edges and degree sequence $d_1, d_2,\dots, d_n$. For each $i\in\{1,\ldots,n\}$, if the neighborhood of the vertex $i$ contains at most $\epsilon d_i^{3/2}$ edges for some constant $\epsilon>0$, then there is a constant $c=c(\epsilon)>0$ such that $G$ has a bisection of size at least
\begin{align}\label{THM:bisection-version-shear-bound}
\frac m2+c\sum_{i=1}^n\sqrt{d_i}-2m\sqrt{\Delta / n},
\end{align}
where $\Delta$ is the maximum degree of $G$.
\end{theorem}

Considering bisections of $K_{k, n}$, the term $m\sqrt{\Delta / n}$ in Theorem \ref{max bisection} is essential and cannot be omitted. The proof of Theorem \ref{max bisection} is presented in Section \ref{SEC:pf-thm-sparse}.



\section{Preliminaries}\label{SEC:Preliminaries}

Let $G$ be a graph. We use $v(G)$ and $e(G)$ to denote the number of vertices and edges of $G$, respectively. For a vertex $v\in V(G)$, let $N_G(v)$ denote the \emph{neighborhood} of $v$ and $d_G(v)=|N_G(v)|$ be the \emph{degree} of $v$. Let $\Delta(G)$ be the \emph{maximum degree} of $G$, and $\delta(G)$ be the \emph{minimum degree} of $G$, respectively. The \emph{average degree} of $G$ is $d(G)=2e(G)/v(G)$.  For $S\subseteq V(G)$, let $G[S]$ denote the subgraph induced by $S$ and $e_G(S)$ be the number of edges in $G[S]$. Hereafter, we omit the subscript $G$ for brevity.

We now present several key results that will be essential for our proofs. First, we recall some fundamental results about the number of edges in $C_{2k}$-free graphs.


\begin{theorem}[Bondy and Simonovits \cite{BS1974}]\label{c2k turan}
    Let $k\ge2$ be an integer and let $G$ be a graph on $n$ vertices. If $G$ contains no
cycle of length $2k$, then the number of edges of $G$ is at most $100kn^{1+1/k}$.
\end{theorem}

\begin{theorem}[Naor and Verstra\"ete \cite{NV2005}]\label{c2k bipartite turan}
    Let $k\ge2$ be an integer and let $G$ be an $a$ by $b$ bipartite graph with $a\le b$. If $G$ is $C_{2k}$-free, then
    \begin{align*}
    e(G)\le\begin{cases}
        2k\left((ab)^{\frac{k+1}{2k}}+a+b\right),&\text{if $k$ is odd};\\
        2k\left(a^{\frac{k+2}{2k}}b^{\frac12}+a+b\right),&\text{if $k$ is even}.
    \end{cases}
    \end{align*}
\end{theorem}

Next, we list several technical lemmas on Max-Bisections. The following is a direct corollary of a result given by Lee, Loh and Sudakov \cite{LLS2013}.
\begin{theorem}[Lee, Loh and Sudakov \cite{LLS2013}]\label{Lee Loh Sudakov JCTB 13}
Every graph with $n$ vertices, $m$ edges,  maximum degree $\Delta$ and without isolated vertices has a bisection of size at least
\begin{align*}
\frac m2 +\frac{n - \max\{n/3, \Delta - 1\}}4 .
\end{align*}
\end{theorem}

Hou and Yan \cite{HY2020} established an improved lower bound for Max-Bisections of $C_4$-free graphs, strengthening a previous result of Fan, Hou and Yu \cite{FHY2018}.

\begin{theorem}[Hou and Yan \cite{HY2020}]\label{Fan Hou Yu CPC18}
Every connected $C_4$-free graph $G$ with $n$ vertices, $m$ edges and minimum degree at least 2 admits a bisection of size at least $m/2+(n-1)/4.$
\end{theorem}

We end this section with the following intuitive lemma.

\begin{lemma}\label{max bisection subgraph}
Let $G$ be a graph with $m$ edges and $S\subseteq V(H)$. If $G[S]$ has a bisection of size $e(S)/2+x$, then $G$ has a bisection with size at least $m/2+x-\sqrt{8m}$.
\end{lemma}
    \begin{proof}[\bf Proof]
    Denote $T=V(G)\setminus S$. Let $(A_1,B_1)$ be a bisection of $G[S]$ with size $e(G[S])/2+x$. By Theorem \ref{Lee Loh Sudakov JCTB 13}, $G[T]$ has a bisection $(A_2, B_2)$ of size at least $e(T)/2$. Note that $e(A_1,B_2)+e(B_1,A_2)+e(A_1,A_2)+e(B_1,B_2)=e(S,T)$. One of $e(A_1,B_2)+e(B_1,A_2)$ and $e(A_1,A_2)+e(B_1,B_2)$ has value at least $e(S,T)/2$, say
    \begin{align*}
    e(A_1,B_2)+e(B_1,A_2)\ge e(S,T)/2.
    \end{align*}
    Let $A=A_1\cup A_2$ and $B=B_1\cup B_2$. Then $(A,B)$ is a bipartition of $G$ with size at least
    \begin{align*}
    e(S)/2+x+e(T)/2+e(S,T)/2\ge m/2+x
    \end{align*}
    and $\left||A|-|B|\right|\le2$. If $\left||A|-|B|\right|\le1$, then we are done. Otherwise, we pick a vertex with degree at most $\sqrt{8m}$ from the larger part to the other part and we obtain a bisection whose size decreases at most $\sqrt{8m}$ compared to $e(A,B)$. It is possible as there are at least $n/2$ vertices whose degree is at most $\sqrt{8m}$. This completes the proof of the lemma.
    \end{proof}

\section{Proof of Theorem \ref{max bisection} }\label{SEC:pf-thm-sparse}

\subsection{A technical lemma}\label{Semi-definite programming}

Inspired by the celebrated approximation algorithm of Goemans and Williamson \cite{GW1995}, Carlson, Kolla, Li, Mani, Sudakov and Trevisan \cite{CKL2021} proposed an intuitive approach to study the Max-Cut of sparse $H$-free graphs by semidefinite programming. The following is crucial in their analysis.

\begin{lemma}[Glock, Janzer and Sudakov \cite{Glock2023}]\label{LEM:Max-Cut-sdp-emma}
Let $G$ be a graph with $m$ edges and let $N$ be a positive integer.
Then, for any set of non-zero vectors $\{\mathbf x^v:v\in V(G)\}\subset\mathbb R^N$, $G$ has a cut of size at least
\[
\frac{m}{2}-\frac1{\pi}\sum_{uv\in E(G)}\arcsin\left(\frac{\langle\mathbf x^u,\mathbf x^v\rangle}{\|\mathbf x^u\|\|\mathbf x^v\|}\right).
\]
\end{lemma}

Motivated by a recent work of R\"aty and Tomon \cite{RT2024}, we present a bisection variant of Lemma \ref{LEM:Max-Cut-sdp-emma}.

\begin{lemma}\label{sdp lemma}
Let $G$ be a graph with $m$ edges and average degree $d$, and let $N$ be a positive integer. Then, for any set of unit vectors $\{\mathbf y_v:v\in V(G)\}\subset\mathbb R^N$, $G$ has a bisection of size at least
\begin{align*}
\frac m2-\frac1{\pi}\sum_{uv\in E(G)}\arcsin(\langle\mathbf y_u,\mathbf y_v\rangle)-d\sqrt{n+\frac4{\pi}\sum_{\{u,v\}\in {V\choose 2}}\arcsin(\langle\mathbf y_u,\mathbf y_v\rangle)}.
\end{align*}

\end{lemma}
\begin{proof}[\bf Proof]

Suppose that $G$ is a graph with $n$ vertices, $m$ edges and average degree $d$. For convenience, set $V\coloneqq V(G)$.  Let $\mathbf w$ be a uniformly distributed random unit vector in $\mathbb R^{N}$.  Let $X=\{v\in V:\langle\mathbf w,\mathbf y_v\rangle\ge0\}$ and $Y=V\setminus X$. Then $(X,Y)$ be a random bipartition of $G$. Note that there are at least $n/2$ vertices with degree at most $2d$. So, we can choose $||X|-\lfloor n/2\rfloor|$ vertices with degrees at most $2d$ in the part with larger size, and move them to part with smaller size. This results in a random bisection $(A, B)$ of $G$ with $e(A,B)\ge e(X,Y)-d\left|2|X|-n\right|$. Thus,
\begin{align}\label{EQ:ex-e(A,B)-semi}
\mathbb E[e(A,B)]\ge \mathbb E[e(X,Y)]-d\;\mathbb E[|2|X|-n|].
\end{align}

We now proceed to evaluate each term on the right-hand side of  \eqref{EQ:ex-e(A,B)-semi}.  For a pair $\{u,v\}\in {V\choose 2}$, we have
\begin{align}\label{EQ:pro-uv-same-part-semi}
\text{Pr}[u,v\in X]=\text{Pr}[u,v\in Y]=\frac{\pi-\arccos(\langle\mathbf y_u,\mathbf y_v\rangle)}{2\pi}=\frac14+\frac1{2\pi} \arcsin(\langle\mathbf y_u,\mathbf y_v\rangle)
\end{align}
    and then
\begin{align*}
\text{Pr}[\text{$u,v$ belong to different parts under $(X,Y)$}]=\frac12-\frac1{\pi} \arcsin(\langle\mathbf y_u,\mathbf y_v\rangle).
\end{align*}
By the linearity of expectation, we have
\begin{align}\label{EQ:exp-e(X,Y)-semi}
\mathbb E[e(X,Y)]=\sum_{uv\in E(G)}\text{Pr}[\text{$uv$ is a cut edge of $(X,Y)$}]=\frac {m}2-\frac1{\pi}\sum_{uv\in E(G)}\arcsin(\langle\mathbf y_u,\mathbf y_v\rangle).
\end{align}

To bound $\mathbb E[|2|X|-n|]$,  notice that $\mathbb E[|X|]=n/2$ and
\begin{align}\label{EQ:upper-bound-exp-2|X|-n-semi}
\mathbb E[|2|X|-n|]\le\sqrt{\mathbb E(2|X|-n)^2}=\sqrt{4\mathbb E[|X|^2]-4\mathbb E[|X|]n+n^2}=\sqrt{4\mathbb E[|X|^2]-n^2}.
\end{align}
For $v\in V$, let $\textbf{1}_{v\in X}$ be the indicator random variable of the event $v\in X$. By \eqref{EQ:pro-uv-same-part-semi}, we have
    \begin{align*}
     \mathbb E[|X|^2]&=\mathbb E\left[\left(\sum_{v\in V}\textbf{1}_{v\in X}\right)^2\right]=\sum_{(u,v)\in V^2}\mathbb E[\textbf{1}_{u\in X}\textbf{1}_{v\in X}]\\
        &=\frac{n}{2}+\sum_{(u, v)\in  {V\choose 2}}\text{Pr}[u,v\in X]\\
        &=\frac{n}{2}+2\sum_{\{u,v\}\in {V\choose 2}}\left(\frac14+\frac1{2\pi}\arcsin(\langle\mathbf y_u,\mathbf y_v\rangle)\right)\\
        &=\frac{n^2}4+\frac n4+\frac1{\pi}\sum_{\{u,v\}\in {V\choose 2}}\arcsin(\langle\mathbf y_u,\mathbf y_v\rangle).
    \end{align*}
Substituting the above into  \eqref{EQ:upper-bound-exp-2|X|-n-semi} gives
    \begin{align*}
    \mathbb E[|2|X|-n|]\le\sqrt{n+\frac4{\pi}\sum_{\{u,v\}\in {V\choose 2}}\arcsin(\langle\mathbf y_u,\mathbf y_v\rangle)}.
    \end{align*}
This together with  \eqref{EQ:ex-e(A,B)-semi} and \eqref{EQ:exp-e(X,Y)-semi} establishes that
    \begin{align*}
    \mathbb E[e(A,B)]\ge\frac m2-\frac1{\pi}\sum_{uv\in E(G)}\arcsin(\langle\mathbf y_u,\mathbf y_v\rangle)-d\sqrt{n+\frac4{\pi}\sum_{\{u,v\}\in {V\choose 2}}\arcsin(\langle\mathbf y_u,\mathbf y_v\rangle)},
    \end{align*}
    completing the proof.
\end{proof}

\subsection{Proof of Theorem \ref{max bisection}}
In this subsection, we present a proof of Theorem \ref{max bisection} using Lemma \ref{sdp lemma}.
Let $G$ be a graph with $n$ vertices, $m$ edges and the maximum degree $\Delta$. Suppose that for every vertex $v\in V(G)$, we have $e(G[N(v)])\le C d(v)^{3/2}$.
Let $\gamma>0$ (to be chosen later). For $v\in V(G)$, define
\begin{align*}
\mathbf x_v(u)=
\begin{cases}
    1,&u=v;\\
    -\frac{\gamma}{\sqrt{d(v)}},&u\in N(v);\\
    0,&\text{otherwise.}
\end{cases}
\end{align*}
Then $\|\mathbf x_v\|^2=1+\gamma^2$. Let $d(u,v)=|N(u)\cap N(v)|$. For $\{u,v\}\in {V\choose 2}$, if $uv\in E(G)$, then
\begin{align}\label{EQ:inner-xu-xv-edge}
\langle\mathbf x_u,\mathbf x_v\rangle=-\frac\gamma{\sqrt{d(u)}}-\frac\gamma{\sqrt{d(v)}}+\frac{\gamma^2 d(u,v)}{\sqrt{d(u)d(v)}};
\end{align}
otherwise,
\begin{align}\label{EQ:inner-xu-xv-nonedge}
\langle\mathbf x_u,\mathbf x_v\rangle=\frac{\gamma^2 d(u,v)}{\sqrt{d(u)d(v)}}.
\end{align}
Let
\[
\mathbf y_v=\frac{\mathbf x_v}{\|\mathbf x_v\|}=\frac{\mathbf x_v}{\sqrt{1+\gamma^2}}.
\]
Note that for any $x\in[-1,1]$ with $x\le b-a$ for some $a,b\ge0$, we have $\arcsin(x)\le\frac{\pi}2b-a$. This is because if $x < 0$, then $\arcsin(x)\le x\le b-a\le\frac{\pi}2b-a$; otherwise, $\arcsin(x)\le\frac{\pi}2x\le\frac{\pi}2(b-a)\le\frac{\pi}2b-a$.
If $uv\in E(G)$, then by \eqref{EQ:inner-xu-xv-edge}
\begin{align}\label{sum uvE}
\arcsin(\langle\mathbf y_u,\mathbf y_v\rangle)\le-\frac{\gamma}{1+\gamma^2}\left(\frac1{\sqrt{d(u)}}+\frac1{\sqrt{d(v)}}\right)+\frac{\pi \gamma^2d(u,v)}{2(1+\gamma^2)\sqrt{d(u)d(v)}};
\end{align}
otherwise, by \eqref{EQ:inner-xu-xv-nonedge}
\begin{align}\label{sum uvV}
\arcsin(\langle\mathbf y_u,\mathbf y_v\rangle)\le \frac{\pi \gamma^2d(u,v)}{2(1+\gamma^2)\sqrt{d(u)d(v)}}.
\end{align}
The summation of  $\frac{d(u,v)}{\sqrt{d(u)d(v)}}$ across all edges $uv$ in $G$ gives
\begin{align*}
\sum_{uv\in E(G)}\frac{d(u,v)}{\sqrt{d(u)d(v)}}&\le\frac12\sum_{uv\in E(G)}\left(\frac{d(u,v)}{d(u)}+\frac{d(u,v)}{d(v)}\right)\\
&=\frac12\sum_{u\in V(G)}\sum_{v\in N(u)}\frac{d(u,v)}{d(u)}\\
&=\sum_{u\in V(G)}\frac{e(G[N(u)])}{d(u)}.
\end{align*}
Recall that $e(G[N(u)])\le C d(u)^{3/2}$. Consequently,
\begin{align}\label{EQ:sum-edge-degree-semi}
\sum_{uv\in E(G)}\frac{d(u,v)}{\sqrt{d(u)d(v)}}\le C\sum_{v\in V(G)}\sqrt{d(v)}.
\end{align}
Similarly, summing  $\frac{d(u,v)}{\sqrt{d(u)d(v)}}$ over all pairs $uv$ gives
\begin{align*}
\sum_{\{u,v\}\in {V\choose 2}}\frac{d(u,v)}{\sqrt{d(u)d(v)}}\le\frac12\sum_{u\in V(G)}\sum_{v\in V(G)\setminus\{u\}}\frac{d(u,v)}{d(u)}.
\end{align*}
Observe that the sum $\sum_{v\in V(G)\setminus\{u\}}d(u,v)$ enumerates all paths of length 2 with an endpoint $u$. Thus,
\begin{align*}
\sum_{v\in V(G)\setminus\{u\}}d(u,v)=\sum_{w\in N(u)}(d(w)-1),
\end{align*}
and then
\begin{align}\label{EQ:sum-nonedge-degree-semi}
\sum_{\{u,v\}\in {V\choose 2}}\frac{d(u,v)}{\sqrt{d(u)d(v)}}&\le\frac12\sum_{u\in V(G)}\sum_{w\in N(u)}\frac{d(w)-1}{d(u)}\notag\\
&\le\frac12n(\Delta-1).
\end{align}
Combining \eqref{sum uvE}, \eqref{sum uvV}, \eqref{EQ:sum-edge-degree-semi} and \eqref{EQ:sum-nonedge-degree-semi}, we conclude that
\begin{align*}
\sum_{uv\in E(G)}\arcsin(\langle\mathbf y_u,\mathbf y_v\rangle)&\le
-\frac{\gamma}{1+\gamma^2}\sum_{v\in V(G)}\sqrt{d(v)}+\frac{C\pi}2\frac{\gamma^2}{1+\gamma^2}\sum_{v\in V(G)}\sqrt{d(v)},
\end{align*}
and
\begin{align*}
    \sum_{\{u,v\}\in {V\choose 2}}\arcsin(\langle\mathbf y_u,\mathbf y_v\rangle)&\le\frac\pi2\frac{\gamma^2}{1+\gamma^2}\sum_{\{u,v\}\in {V\choose 2}}\frac{d(u,v)}{\sqrt{d(u)d(v)}}\\
    &\le\frac\pi4\frac{\gamma^2}{1+\gamma^2}n(\Delta-1).
\end{align*}
So by Lemma \ref{sdp lemma}, $G$ has a bisection with size at least
\begin{align*}
\frac m2+\frac{\gamma}{1+\gamma^2}\left(\frac1{\pi}-\frac C 2\gamma\right)\sum_{v\in V(G)}\sqrt{d(v)}-d\sqrt{\frac{n}{1+\gamma^2}+\frac{\gamma^2}{1+\gamma^2}n\Delta}.
\end{align*}
Choose $\gamma>0$ such that $\frac1{\pi}-\frac C2\gamma\ge\frac1{2\pi}$, completing the proof.

\section{Bisections of $C_{2k}$-free graphs}\label{Bisections of C_2k-free graphs}
In this section, we investigate  bisections of  $C_{2k}$-free graphs for $k\ge3$ and establish the following theorem.

\begin{theorem}\label{THM:Max-bisection-C2k-free-tight}
Let $k\ge3$ be an integer, and let $G$ be a $C_{2k}$-free graph with $m$ edges and minimum degree at least $k$. Then  $G$ has a bisection of size at least
\begin{align*}
m/2+\Omega(m^{(2k+1)/(2k+2)}).
\end{align*}
\end{theorem}

First, we derive the following intuitive result using Theorem \ref{max bisection} and a standard degenerate argument.
\begin{lemma}\label{max bisection with max degree}
Let $k\ge2$ be an integer, and let $G$ be a $C_{2k}$-free graph with $n$ vertices,  $m$ edges and maximum degree $\Delta$. If there is a positive constant $c$ such that $\Delta\le cn/m^{1/(k+1)}$, then $G$ has a bisection of size at least
    \begin{align*}
    m/2+\Omega\left(m^{(2k+1)/(2k+2)}\right).
    \end{align*}
\end{lemma}
\begin{proof}[\bf Proof]
Let $G$ be a $C_{2k}$-free graph with $n$ vertices,  $m$ edges and maximum degree $\Delta$. We may assume that $n$ is sufficiently large. For each $v\in V(G)$, we know that $G[N(v)]$ does not contain a path of length $2k-2$ by the freeness of $C_{2k}$. By the classical Erd\H{o}s-Gallai Theorem, we have $e(N(v))\le kd(v)$. Thus, by Theorem \ref{max bisection}, there is a constant $c_1$ such that $G$ has a bisection of size at least
\begin{align}\label{EQ:bisection-C2k-small-maximum-degree}
m/2+c_1\sum_{v\in V(G)}\sqrt{d(v)}-2m\sqrt{\Delta/n}.
\end{align}

The crux of the proof lies in demonstrating
\begin{align}\label{EQ:degree-sequence-small-maximum-degree}
\sum_{v\in V(G)}\sqrt{d(v)}\ge\frac{1}{\sqrt{200k}}m^{\frac{2k+1}{2k+2}}.
\end{align}
Define $L=200km^{1/(k+1)}$. We claim that  $G$ is $(L-1)$-degenerate. Otherwise,  $G$ contains a subgraph $G'$ with minimum degree at least $L$. Note that the number of vertices of $G'$ is $N \leq \frac{2m}L<m^{k/(k+1)}$. Thus, $L> 200kN^{1/k}$ and then
\[e(G')\ge \frac{LN}{2}> 100kN^{1+1/k},\]
which is impossible by Theorem \ref{c2k turan}.  By the degeneracy of $G$, we can label the vertices of $G$ by $v_1,\ldots,v_n$ so that for every $i$, the number of neighbours $v_j$ of $v_i$ with $j<i$ is strictly smaller than $L$. Let $d^+(v)$ be the number of neighbours $v_j$ of $v_i$ with $j<i$. Then
\begin{align*}
\sum_{v\in V(G)}\sqrt{d(v)}\ge\sum_{v\in V(G)}\sqrt{d^+(v)}=\sum_{v\in V(G)}\frac{d^+(v)}{\sqrt{d^+(v)}}\ge\sum_{v\in V(G)}\frac{d^+(v)}{\sqrt{L}}=\frac{1}{\sqrt{200k}}m^{\frac{2k+1}{2k+2}}.
\end{align*}
Set $c:=\frac{c_1^2}{3200k}$. We conclude that
\begin{align*}
2m\sqrt{\Delta/n}\le \frac{c_1}{2\sqrt{200k}}m^{\frac{2k+1}{2k+2}}\le \frac{c_1}2\sum_{v\in V(G)}\sqrt{d(v)},
\end{align*}
which together with \eqref{EQ:bisection-C2k-small-maximum-degree} and \eqref{EQ:degree-sequence-small-maximum-degree} establishes that
$G$ has a bisection of size at least
\begin{align*}
\frac m2+\frac{c_1}{2\sqrt{200k}}m^{\frac{2k+1}{2k+2}}.
\end{align*}
Thus, we complete the proof of this lemma.
\end{proof}

\begin{proof}[\bf Proof of Theorem \ref{THM:Max-bisection-C2k-free-tight}.]

Let $G$ be a $C_{2k}$-free graph with $n$ vertices, $m$ edges and minimum degree at least $k$. Suppose that $n$ is sufficiently large.  By Theorem \ref{c2k turan},  we have
\begin{align}\label{EQ:bound-average-degree-turan}
m\le100kn^{1+1/k}.
\end{align}
Let $c>0$ be a constant given by  Lemma \ref{max bisection with max degree}, and let
\[
D:=cn/(2m)^{1/(k+1)}.
\]
By Lemma \ref{max bisection with max degree}, we may assume that $\Delta(G)\ge D$. Now, we first show that $G$ is relatively sparse.

\begin{claim}\label{Sparse}
We have $m\le3kn$.
\end{claim}
\begin{proof}
Let $A=\{v\in V(G):d(v)\ge D\}$ and $B=V(G)\setminus A$. Recall that $k\ge 3$. By \eqref{EQ:bound-average-degree-turan}, we have
    \begin{align}\label{EQ:bound-size-A-C2k-tight}
    |A|\le \frac{2m}{D}=\frac{(2m)^{(k+2)/(k+1)}}{cn}=O(n^{2/k})=o(n).
    \end{align}
This together with Theorem \ref{c2k turan} implies that
\begin{align}\label{EQ:bound-number-edge-A-C2k}
e(A)\le100k|A|^{(k+1)/k}=O\left(n^{(2k+2)/k^2}\right)=o(n).
\end{align}
Moreover, we can conclude that for any constant $\epsilon>0$
\begin{align}\label{B-sparse}
e(B)\le\epsilon m.
\end{align}
Otherwise, we may assume that $e(B)\ge\eta m$ for some $\eta>0$. Note that $G[B]$ has maximum degree at most $D$ and $ D\le c|B|/e(B)^{1/(k+1)}$ as $|B|=(1-o(1))n$ by \eqref{EQ:bound-size-A-C2k-tight}. Applying Lemma \ref{max bisection with max degree} to $G[B]$ results in a bisection of $G[B]$ with size at least $e(B)/2+\Omega(e(B)^{(2k+1)/(2k+2)})$. By Lemma \ref{max bisection subgraph}, there is a bisection of $G$ with size at least
        \begin{align*}
        m/2+\Omega\left(e(B)^{(2k+1)/(2k+2)}\right)-\sqrt{8m}= m/2+\Omega\left(m^{(2k+1)/(2k+2)}\right).
        \end{align*}
This implies that $G$ has a large bisection as desired. Thus, we know that \eqref{B-sparse} holds.

Combining \eqref{EQ:bound-number-edge-A-C2k} and \eqref{B-sparse}, we deduce that
 \begin{align}\label{Assume:e(A,B)>=0.99m}
e(A, B)=m-e(A)-e(B)=(1-o(1))m.
\end{align}
Now, we establish an upper bound of  $e(A, B)$ by Theorem \ref{c2k bipartite turan}. If $k$ is even, then $k\ge 4$. This together with \eqref{EQ:bound-size-A-C2k-tight} and Theorem \ref{c2k bipartite turan} yields that
    \begin{align*}
    e(A, B)\le2k\left(|A|^{\frac{k+2}{2k}}n^{\frac12}+n\right)=2kn+O\left(n^{\frac{k^2+2(k+2)}{2k^2}}\right)=(2k+o(1))n.
    \end{align*}
Otherwise, using \eqref{EQ:bound-size-A-C2k-tight} and Theorem \ref{c2k bipartite turan} again for odd $k$ gives that
\begin{align*}
e(A, B)\le2k\left(\left(|A|n\right)^{\frac{k+1}{2k}}+n\right)=2kn+O\left(m^{\frac{k+2}{2k}}\right)=2kn+o(m).
\end{align*}
It follows from \eqref{Assume:e(A,B)>=0.99m} that $m\le 3kn$ in either case. This completes the proof of the claim.
\end{proof}

In what follows, it suffices to show that $G$ contains a large induced subgraph without isolated vertices whose maximum degree is far away from its number of vertices. This together with Theorem \ref{Lee Loh Sudakov JCTB 13} and Lemma \ref{max bisection subgraph} would finish our proof.
Let
\[
S=\left\{v\in V(G):d(v)\ge\left(1-\frac1{2k}\right)n\right\} \text{ and } T=V(G)\setminus S.
\]
Note that $|S|\le k-1$. Since otherwise,  for any $k$-subset $S'\subseteq S$,
    \begin{align*}
       \left |\bigcap_{v\in S'}N(v)\right|
       \ge n-\sum_{v\in S'}\left|\overline{N(v)}\right|
       \ge n-k\left(n-\left(1-\frac1{2k}\right)n\right)=\frac n2.
    \end{align*}
This means that  $G$ contains a copy of $K_{k,n/2}$, contradicting the freeness of $C_{2k}$.  Recall that the minimum degree of $G$ is at least $k$. This together with the fact $|S|\le k-1$ implies that
    \begin{align*}
    d_{G[T]}(v)  \ge d(v)-|S|\ge 1
    \end{align*}
for each $v\in T$. Thus, $G[T]$ has maximum degree at most $(1-\frac1{2k})n$ and minimum degree at least one. It follows from Theorem \ref{Lee Loh Sudakov JCTB 13} that $G[T]$ has a bisection of size at least $e(T)/2+\Omega(|T|)$. Consequently, by Lemma \ref{max bisection subgraph}, $G$ has a bisection of size at least
    \begin{align*}
    m/2+\Omega(|T|)-\sqrt{8m}\ge m/2+\Omega(m),
    \end{align*}
    where the last inequality follows from Claim \ref{Sparse} and the fact $|T|=n-|S|=\Omega(n)$. Thus, we obtain a much larger cut than expected, and complete the proof of Theorem \ref{THM:Max-bisection-C2k-free-tight}.
\end{proof}
\section{Bisections of $C_4$-free graphs}\label{Bisections of C_4-free graphs}
In this section, we prove Theorem \ref{max bisection for c2kfree} for $C_4$-free graphs as follows.

\begin{theorem}\label{max bisection for c4free}
Every $C_{4}$-free graph with $m$ edges and minimum degree at least $2$ has a bisection of size at least $m/2+\Omega(m^{5/6}).$
\end{theorem}

Following the strategy in the proof of Theorem \ref{THM:Max-bisection-C2k-free-tight}, we first prove the following theorem.
\begin{theorem}\label{max bisection for c4free sparse}
There exists an absolute positive constant $\ell$ such that the following holds. Let $G$ be a $C_{4}$-free graph with $n$ vertices and $m$ edges.  Suppose further that $n^{6/5}\le m\le\ell n^{3/2}$. Then $G$ has a bisection of size at least $m/2+\Omega(m^{5/6})$.
\end{theorem}
\begin{proof}[\bf Proof]
Let $G$ be a $C_{4}$-free graph with $n$ vertices and $m$ edges. We may assume that $n$ is sufficiently large and the constant $c$ from Lemma \ref{max bisection with max degree} is small enough (say $c<1$). Set $\ell :=\left(\frac{c}{200}\right)^3$ and define
\[
D:=\frac{cn}{(2m)^{1/3}}.
\]
Clearly,   we may assume that $\Delta(G)\ge D$ by Lemma \ref{max bisection with max degree}. In what follows, we aim to find a large induced subgraph of $G$ with $\Omega(m)$ edges.

Let $A=\{v\in V(G):d(v)\ge D\}$ and $B=V(G)\setminus A$. Since $G$ is $C_4$-free, it follows from Theorems \ref{c2k turan} and \ref{c2k bipartite turan} that
\begin{align}\label{AB}
e(A)\le200|A|^{3/2} \;\text{and}\; e(A,B)\le4\left(|A|\sqrt n+n\right).
\end{align}
Note that $n^{6/5}\le m\le\ell n^{3/2}$. This together with the definitions of $D$ and $\ell $ yields that
\[
D=\frac{cn}{(2m)^{1/3}}\ge\frac{cn}{(2\ell  n^{3/2})^{1/3}}\ge100\sqrt n
\]
and
\[
\frac{2m}{D}=\frac{(2m)^{4/3}}{cn}\ge\frac{(2n^{6/5})^{4/3}}{cn}=\Omega(n^{3/5}).
\]
Thus, we have
\[
|A|\le\frac{2m}{D}\le\frac{2\ell  n^{3/2}}{100\sqrt n}<\frac{n}{100} \;\text{and}\; \sqrt n=o\left(\frac{2m}{D}\right).
\]
This together with \eqref{AB} deduce that
\begin{align*}
e(A)+e(A,B)
& \le\sqrt n\left(20|A|+4|A|+4\sqrt n\right)\\
& \le\sqrt n\left(24\cdot\frac{2m}{D}+o\left(\frac{2m}{D}\right)\right)\\
& \le\frac m2.
\end{align*}
Now, we conclude that $G[B]$ has at least $0.99n$ vertices and $m/2$ edges. Note also that $G[B]$ has maximum degree at most $D$ and $ D\le c|B|/e(B)^{1/3}$ as $|B|\ge 0.99n$. Applying Lemma \ref{max bisection with max degree} to $G[B]$ results in a bisection of $G[B]$ with size at least $e(B)/2+\Omega(e(B)^{5/6})$. By Lemma \ref{max bisection subgraph}, there is a bisection of $G$ with size at least
        \begin{align*}
        m/2+\Omega\left(e(B)^{5/6}\right)-\sqrt{8m}= m/2+\Omega\left(m^{5/6}\right).
        \end{align*}
This implies that $G$ has a large bisection as desired, and completes the proof of this theorem.
\end{proof}

\begin{proof}[\bf Proof of Theorem \ref{max bisection for c4free}]
Let $G$ be a $C_{4}$-free graph with $n$ vertices, $m$ edges and minimum degree at least $2$. Suppose that $n$ is sufficiently large. Let $\ell $ be the constant from Theorem \ref{max bisection for c4free sparse}. If $m\le n^{6/5}$, then, by Theorem \ref{Fan Hou Yu CPC18}, $G$ has a bisection of size at least
    \begin{align*}
    \frac m2 + \frac{n - 1}4\ge \frac m2+\Omega(m^{5/6}),
    \end{align*}
indicating the desired result. If $n^{6/5}\le m\le\ell n^{3/2}$, then we are done by Theorem \ref{max bisection for c4free sparse}. In what follows, we may assume that  $m\ge\ell n^{3/2}$, and we aim to find an induced subgraph $G'$ of $G$ with $\Omega(m)=e(G')\le \ell |V(G')|^{3/2}$.

Define $G_0=G$ and let $G_i$ be the graph obtained by removing a vertex with its degree $\Delta(G_{i-1})$ from $G_{i-1}$  for each $i\in\{1,\ldots,n-1\}$. For simplicity, let
    \begin{align*}
        m_i=e(G_i),\Delta_i=\Delta(G_i)\text{ and }X_i=\frac{e(G_i)}{|V(G_i)|^{3/2}}=\frac{m_i}{(n-i)^{3/2}}.
    \end{align*}
In particular, write
    \begin{align*}
        X_0:=\frac{m}{n^{3/2}}.
    \end{align*}
    Clearly, we have $\ell \le X_0\le200$ by Theorem \ref{c2k turan} and our assumption. Note also that $m_i=m_{i-1}-\Delta_{i-1}$ for each $i\in\{1,\ldots,n-1\}$.
It follows from $\Delta_i\ge2e(G_i)/|V(G_i)|=2m_i/(n-i)$ that
    \begin{align*}
    m_i\le m_{i-1}-\frac{2m_{i-1}}{n-i+1}=\frac{n-i-1}{n-i+1}m_{i-1}.
    \end{align*}
Consequently, for each $i\in\{1,\ldots,n-1\}$
    \begin{align*}
    X_i=\frac{m_i}{(n-i)^{3/2}}&\le\frac{n-i-1}{n-i+1}\cdot\frac{m_{i-1}}{(n-i)^{3/2}}\\
    &=\frac{n-i-1}{n-i}\cdot\sqrt{\frac{n-i+1}{n-i}}\, X_{i-1}\\
    &\le\sqrt{\frac{n-i}{n-i+1}}\,X_{i-1},
    \end{align*}
implying that $\{X_i\}_{0\le i\le n-1}$ is a strictly decreasing sequence. Moreover, we have
  \begin{align}\label{EQ:upper-bound-Xi-C4-free}
    X_i\le\sqrt{\frac{n-i}{n-i+1}}\,X_{i-1}\le\sqrt{\frac{n-i}{n-i+1}\frac{n-i+1}{n-i+2}\cdots\frac{n-1}{n}}\,X_0=\sqrt{\frac{n-i}{n}}\,X_0.
    \end{align}
Set
\[
L:=\left\lceil\left(1-\left(\frac{\ell }{400}\right)^2\right)n\right\rceil.
\]
This together with \eqref{EQ:upper-bound-Xi-C4-free} and $X_0\le200$ implies that
\[
X_L\le\sqrt{\frac{n-L}{n}}\,X_0\le\frac{\ell }{2}.
\]
Since $\{X_i\}_{0\le i\le n-1}$ is a strictly decreasing sequence with $X_0\ge\ell >0=X_{n-1}$, we can pick a minimal integer $N\in(0,n-1)$ such that $X_{N}\le\ell /2$. This further means that $N\le L$ and $X_{N-1}>\ell /2$. By the definitions of $G_i$ and $X_i$, we have
    \begin{align*}
        \frac{\ell }{2}(n-N+1)^{3/2}-n< e(G_{N-1})-n\le e(G_N)\le \frac{\ell }{2}(n-N)^{3/2}.
    \end{align*}
Consequently,
    \begin{align*}
    e(G_N)=\left(\frac{\ell }{2}+o(1)\right)(n-N)^{3/2}\ge\left(\frac{\ell }{2}+o(1)\right)(n-L)^{3/2}\ge\left(\left(\frac{\ell }{400}\right)^4+o(1)\right)m,
    \end{align*}
    where the last inequality follows from the definition of $L$ and the fact $m\le200 n^{3/2}$.
Applying Theorem \ref{max bisection for c4free sparse} to $G_N$ results in  a bisection of  size at least
$$e(G_N)/2+\Omega(e(G_N)^{5/6}).$$
Thus, by Lemma \ref{max bisection subgraph}, $G$ has a bisection of size at least
$$m/2+\Omega(e(G_N)^{5/6})-\sqrt{8m}= m/2+\Omega(m^{5/6}),$$
completing the proof of Theorem \ref{max bisection for c4free}.
\end{proof}

We immediately obtain Theorem \ref{max bisection for c2kfree} by combining Theorems \ref{THM:Max-bisection-C2k-free-tight} and \ref{max bisection for c4free}. The tightness of Theorem \ref{max bisection for c2kfree} is from Theorem \ref{THM:AKS-thm}, where for each $k\in\{2,3,5\}$ there exists a $C_{2k}$-free graph $G$ with $m$ edges such that every bipartite subgraph of $G$ has size at most $m/2+O(m^{(2k+1)/(2k+2)})$.

\section{Concluding remarks}

The main contribution of this paper is a bisection version of the Alon-Krivelevich-Sudakov Theorem. Our method relies on semidefinite programming techniques, and appears to be worthy of future exploration.

The Shearer's bound serves as an essential tool for establishing rigorous bounds on Max-Bisections of $H$-free graphs \cite{Lin2021, HWY2024}. The known sufficient conditions for graphs to possess bisections meeting Shearer's bound universally require the exclusion of $C_4$ together with other specific graphs. A compelling open problem is to establish the Shearer-type bound for Max-Bisections for $C_{2k}$-free graphs with large minimum degree, especially for $C_4$-free graphs. On the other hand, Theorem \ref{max bisection} requires refinement, as the third term in \eqref{THM:bisection-version-shear-bound} becomes dominant when $\Delta$ is large. This behavior fundamentally arises from the strict part size constraints inherent in the bisection problem, where both partitions must contain exactly $\lfloor n/2\rfloor$ or $\lceil n/2 \rceil$  vertices. However, the third term in \eqref{THM:bisection-version-shear-bound} is essential and cannot be eliminated by considering bisections of $K_{k,n}$.

A natural extension would be to generalize Theorem \ref{max bisection for c2kfree} to forbidden complete bipartite  subgraphs. For $t\ge s\ge 2$, an accessible result, given by Hou and Wu \cite{Hou2021}, is the existence of a bisection of size at least $m/2+\Omega(n)$ in $K_{s,t}$-free graphs with $n$ vertices, $m$ edges and minimum degree $s$.

\end{document}